\documentclass[11pt,reqno]{amsart}
\pdfoutput=1
\usepackage{geometry}   
\usepackage{graphicx}
\usepackage{amssymb}
\usepackage{epstopdf} 
\usepackage{hyperref}
\usepackage{graphicx, amsmath, amsthm, latexsym, amssymb, amsfonts, epsfig, url, paralist}
\usepackage{xcolor}

\newtheorem{theorem}{Theorem}[section]

\newtheorem{proposition}[theorem]{Proposition}

\theoremstyle{definition}

\title[Realizing the associahedron]{Realizing the associahedron:\\ Mysteries and questions}

\author{Cesar Ceballos}
\address[Cesar Ceballos]{Inst.\ Mathematics, FU Berlin, Arnimallee 2, 14195 Berlin, Germany.}  \email{ceballos@math.fu-berlin.de}

\author{G\"unter M. Ziegler}
\address[G\"unter M. Ziegler]{Inst.\ Mathematics, FU Berlin, Arnimallee 2, 14195 Berlin, Germany.} \email{ziegler@math.fu-berlin.de}

\thanks{
	The first author was supported by DFG 
	via the Research Training Group \emph{Methods for Discrete Structures}
	and by Berlin Mathematical School. 
	The research leading to these results has received funding from the European Research
	Council under the European Union's Seventh Framework Programme (FP7/2007-2013) /
	ERC Grant agreement no.~247029-SDModels”.
We are grateful to Carsten Lange and Paco Santos for many interesting discussions.}

\begin{document}
\maketitle
\begin{abstract}
There are many open problems and some mysteries connected to the realizations of the associahedra as convex polytopes. In this note, we describe three --- concerning special realizations with the vertices on a sphere, the space of all possible realizations, and possible realizations of the multiassociahedron.
\end{abstract}

\section{Introduction}

Realizing the $n$-dimensional
 associahedron as a convex polytope is a non-trivial task: You are given
the \emph{combinatorics} of a polytope, and you are
supposed to produce \emph{geometry}, namely coordinates for
a correct realization -- such that the vertices  correspond to the triangulations of an $(n+3)$-gon,
and the facets to its diagonals, and a vertex lies on a facet if and only if the triangulation uses the diagonal.

The realization problem appeared first in Tamari's thesis from 1951 \cite{Z_Tamari1951}. It was explicitly
posed by Stasheff's 1963 paper \cite{Z_St63}, and first
solved somewhat ``by hand'': As far as we know, the $n$-dimensional associahedron was constructed
\begin{compactitem}[--] 
	\item 1963 by Stasheff \cite{Z_St63} as a cellular ball,
	\item 1960s by Milnor for the first time as a polytope (lost),
	\item 1978 by Huguet \& Tamari (see \cite{Z_HuguetTamari1978}: no proof given),
	\item 1984 by Haiman (unpublished, but see \cite{Z_Ha84}), and finally
	\item 1989 by Lee (the first published realization: \cite{Z_Lee2}).
\end{compactitem}

Subsequently, more systematic construction methods emerged,
among them 
\begin{compactitem}[--]
\item	the construction as secondary polytopes of convex $(n+3)$-gons,
\item   the construction from cluster complexes of the root systems $A_n$, and
\item   the construction as a (weighted) Minkowski sum of faces of a simplex,
\end{compactitem}
all of them described in more detail below. In recent work \cite{Z_Z121new}, we have 
discovered that the realizations produced by these three families of
constructions are disjoint, and that they
can be distinguished by quite remarkable, geometric properties -- and moreover,
that there are many more realizations that seem natural as well,
including the exponentially-sized family of Hohlweg \& Lange 
\cite{Z_hohlweg_realizations_2007} \cite[Sect.~4]{Z_Z121new},  and the even larger, 
Catalan-sized family of Santos \cite{Z_Sa04} \cite[Sect.~5]{Z_Z121new}.

There are many open problems and some mysteries connected to the realizations of the associahedra
as convex polytopes. In this note, we describe three: 
\begin{compactitem}[$\circ$] 
\item There are several very natural, but fundamentally 
	different constructions of the $n$-dimensional associahedron, which produce disjoint parameterized
	families of polytopes. How do these families lie in the \emph{realization space} (defined below)
	of the $n$-dimensional associahedron? How do they relate?
\item The associahedron constructed as the secondary polytope of $n+3$ equally-spaced
points on a quadratic planar curve turns out to have all its vertices on an ellipsoid.
This phenomenon extends to the permuto-associahedron and to the cyclohedron.
Explain!

	\item Generalization of triangulations to multitriangulations leads
	to multiassociahedra and, more generally, to generalized multiassociahedra.
	Up to now, one can show that these combinatorial objects are vertex-decomposable
	spheres, but (how) can they be realized as convex polytopes?
\end{compactitem}
Of course this note is written with the hope to clarify the situation and to
explain some observations and pieces of progress related to the problems.
However, some mystery remains, and perhaps this is also natural, in view of 
the sentence that starts Haiman's 1984 manuscript \cite{Z_Ha84}:\\
\centerline{
\includegraphics[width=.7\textwidth]{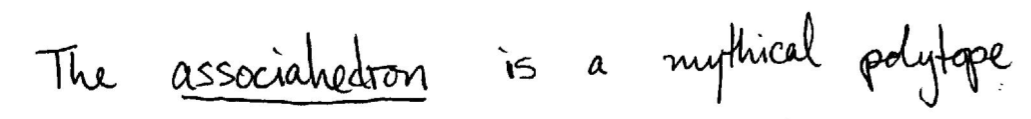}}

\section{Realization space}

As just mentioned, there are three very natural, but fundamentally 
different constructions of the associahedron that may be considered to be ``classical''
by now:
\begin{compactitem}[(I)] 
\item as the secondary polytope of a convex $(n+3)$-gon $Q$ by Gelfand, Kapranov \& Zelevinsky
    \cite{Z_GZK90} \cite{Z_GZK91} (see also \cite[Chap.~7]{Z_GKZ94}),
\begin{equation}\label{Z_def:Secondary}
	{\rm Ass}_n(Q) 
	\ :=\ \text{conv}\{ \sum_{i=1}^{n+3} \sum_{\sigma\in T\,:\, i\in\sigma} \text{vol} (\sigma) f_i :
	 T \text{ is a triangulation of } Q \},
\end{equation}
	where $f_0,\dots,f_{n+2}$ are the vertices of an $(n+2)$-simplex,
\item via cluster complexes of the root system $A_n$ as conjectured by Fomin \& Zelevinsky~\cite{Z_FZ03} and constructed by Chapoton, Fomin \& Zelevinsky \cite{Z_CFZ02}, 
	\[
		{\rm Ass}_n(A_n):= 
		\{ x\in\mathbb{R}^{n+1} \mid x_i-x_j\le f_{i,j}\text{ for } i-j\ge-1,\ \textstyle\sum_i x_i=0\}
\]
	for suitable $f_{i,j}>0$, and
\item as Minkowski sums of simplices, as introduced by Postnikov in \cite{Z_Po05}
	\[
	{\rm Ass}_n(\Delta_n):=
	\textstyle\sum_{1\le i<j\le n}\alpha_{i,j}\Delta_{[i\dots j]},
	\]
	for arbitrary $\alpha_{i,j}>0$,
	which in various different descriptions appears in earlier work by various other authors,
including Shnider \& Sternberg \cite{Z_ShniderSternberg93}, Loday~\cite{Z_Lo04}, Rote \& Santos \& Streinu~\cite{Z_RoSaSt03}.
\end{compactitem}
Some of these realizations have very striking properties, such as
the vertices on a sphere (see below), or having facet normals in the root system $A_n$. 

One would perhaps expect that ``if you set the parameters right‚''
you could get the one-and-only most beautiful realization, but a priori
it is not clear, which one would that be.
However, it turned out (see \cite{Z_Z121new})
that these approaches yield fundamentally distinct realizations.
For example, the associahedra produced as the secondary polytopes
of a convex $(n+3)$-gon don't have \emph{any} parallel facets, like the one in the
figure,
\begin{center}
	\includegraphics[height=28mm]{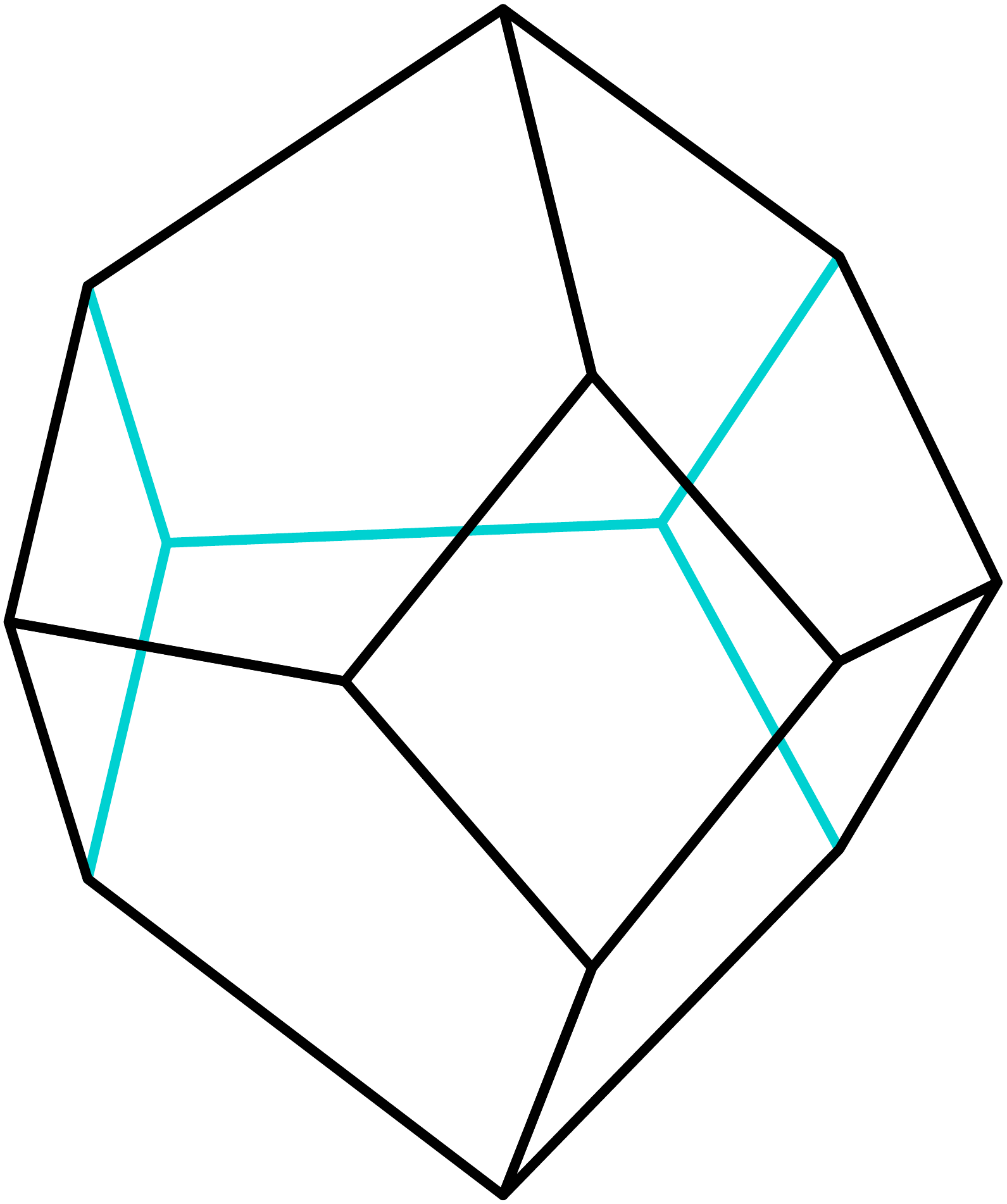}
\end{center}
whereas the others do, typically with $n$ pairs of parallel facets
that correspond to certain pairs of intersecting diagonals:
\begin{center}
	\includegraphics[width=.8\textwidth]{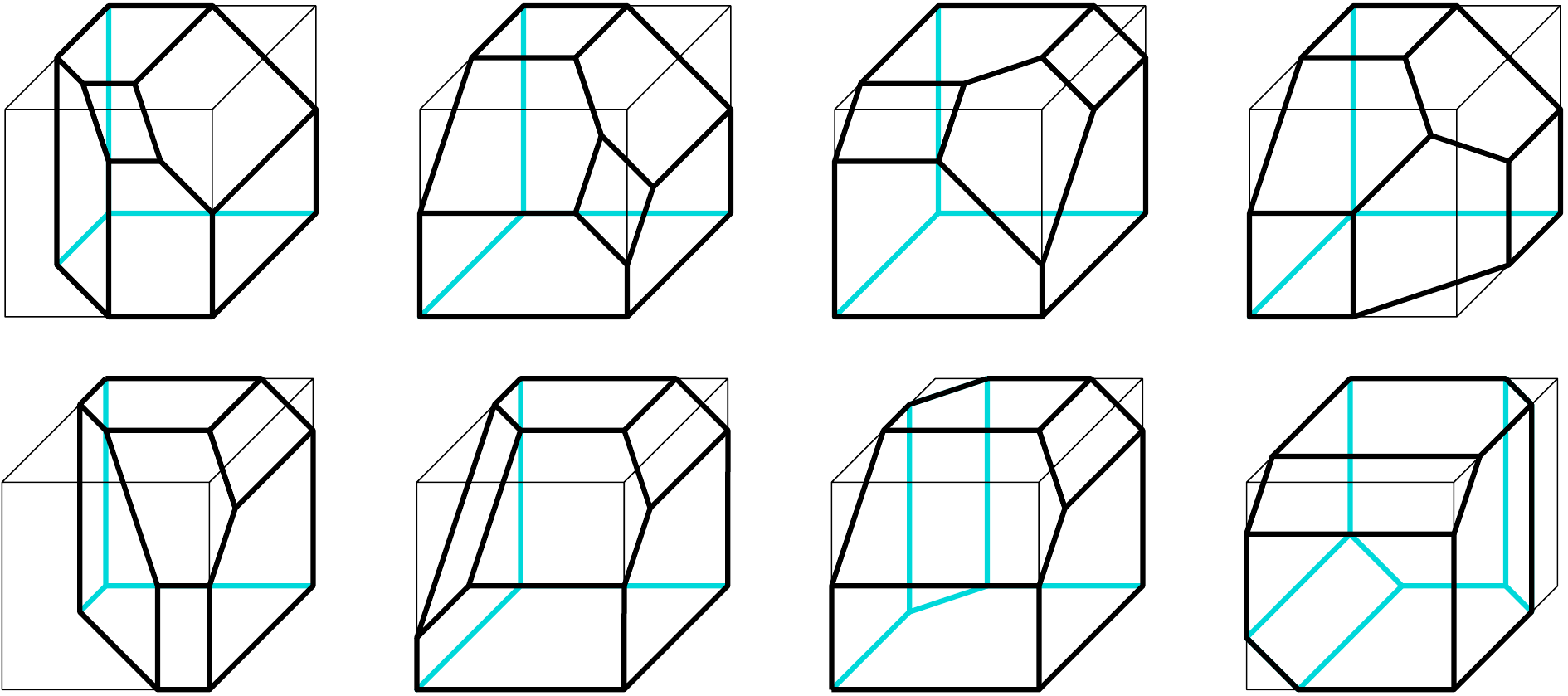}
\end{center}

With the huge number of different realizations 
that are analyzed and distinguished in \cite{Z_Z121new}, one is led to ask
a number of questions about the space of all realizations of
   the $n$-dimensional associahedron:
\begin{compactitem}[$\circ$] 
	\item What is the structure of the space? Is it contractible (if we
	divide out the action of the group of affine transformations, say)?
	Is it even connected?
	\item Do the constructions of associahedra that we know
	   cover a large/typical part of the realization space?
	\item Is there any connection between the realizations?
	Could we get some types as a deformation/limit of other types?
\end{compactitem}
The space of all realizations of (a combinatorial type of) a convex polytope
is known to be a semialgebraic set defined over $\mathbb Z$.
There are various possible definitions, which differ somewhat;
if we do not identify affinely equivalent realizations
and decide to only consider  
realizations with the origin in the interior, then 
the set of all such realizations -- called the \emph{realization space}--
for an $n$-dimensional polytope $P$ with $N$ facets can be identified
with the semialgebraic set 
\[
\{ C\in\mathbb{R}^{d\times f_0}:
\text{ the inequality system }
 C^tx\le1
\text{ defines a realization of }P\}/\text{Aff}(\mathbb R^d).
\]
We refer to Richter-Gebert \cite{Z_Rich4} for an extensive treatment of
realization spaces of polytopes and to \cite{Z_Z45} for an introduction.
The realization space of a simple $n$-dimensional polytope with 
$N=\frac12n(n+3)$ facets can be seen (if we again require the origin to lie in
the interior, and do not divide out a group action)
as an open subset of~$\mathbb{R}^{n\times N}$.
 It is known that realization spaces of some simple polytopes
   are disconnected -- there are sporadic examples in dimension~$4$
(see \cite{Z_BoGO})
and systematic constructions for high-dimensional simple polytopes \cite{Z_JMSW}.
But for the associahedron not much is known beyond dimension $3$,
where Steinitz proved in 1922 (\cite{Z_Stei1}; see~\cite{Z_Rich4})
that the realization space of any $3$-dimensional polytopes
after dividing out the action of the affine group is a topological ball
   of dimension $f_1-6$.

Here is one observation that needs to be followed up: 
The secondary polytope construction produces a realization of the
$n$-dimensional associahedron from any given convex $(n+3)$-gon.
In other words, we get an associahedron from any convex configuration of $n+3$ points in the plane.
The converse to this turns out to be \emph{false}:
``convex position'' is sufficient, but not necessary for getting an associahedron.

\begin{proposition}\label{Z_prop:weakly_convex}
	The secondary polytope of any configuration of $n+3$ points in the plane, which
	consists of all the vertices of a convex polygon and at most one 
	point in the relative interior of any edge, is an $n$-dimensional associahedron.
\end{proposition} 

\begin{proof}
The combinatorial structure of the triangulations of a point configuration $P$
with these properties is exactly the same as the one for a configuration $Q$ 
of points in convex position.
If we cyclically label the vertices of $P$ and $Q$, a triangulation of $P$ corresponds 
to the triangulation of $Q$ consisting of the same diagonals of $P$ together
with the diagonals $i,j$ such that $i,j$ is an edge of the convex hull of $P$ that
has a relative interior point which does no appear on the triangulation. 
\end{proof}
\smallskip

Of course the point configurations of Proposition~\ref{Z_prop:weakly_convex}
are limit cases of strictly-convex configurations, and thus the
realizations of the associahedron obtained from them are
deformations of secondary polytopes of convex $(n+3)$-gons.
However, the deformations do not share all their properties: Indeed,
the following configuration of $6$ points 
\begin{center}
	\includegraphics[width=15mm]{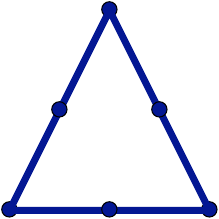}
\end{center}
produces an associahedron that has three pairs of parallel facets -- which
you can never get from a hexagon \cite[Thm.~3.5]{Z_Z121new}.
So, do we get associahedra from cluster algebras,
or associahedra from weighted Minkowski sums, 
as limit cases of secondary polytopes? (For dimension larger than
$3$ we cannot expect that; cf.~\cite[Remark 3.6]{Z_Z121new}.)

And which more general, non-convex planar point configurations
could still produce associahedra?

\section{Vertices on a sphere}

The associahedron constructed as the secondary polytope of $n$ equally-spaced
points on a quadratic planar curve turns out to have all its vertices on an ellipsoid,
or in suitable (and still natural) coordinates even on a sphere:

\begin{theorem}\label{Z_thm:sphericity}
	Let $p,q\in\mathbb{R}[t]$ be quadratic polynomials such that the
	convex curve $C=\{(p(t),q(t))\mid t\in\mathbb{R}\}$ is not a line (that is,
	such that $\{f,g,1\}$ are linearly independent), and let
	$v_0,\ldots,v_{n+2}$ be equally-spaced points on $C$,
	that is, $v_i:=(f(a+ib),g(a+ib))$ for $a,b\in\mathbb{R}$, $a\neq0$.

	Then the secondary polytope of  $Q:=\text{conv}\{v_0,\ldots,v_{n+2}\}$,
	constructed according to {\rm(\ref{Z_def:Secondary})} with
	$f_i:=e_1+\ldots+e_i$, has all its vertices on a sphere around the origin.
\end{theorem}

\begin{proof}
The description/construction of the secondary polytope of a convex polygon
as given here is motivated by the more general setting of
\emph{fiber polytopes} provided by Billera \& Sturmfels \cite{Z_BS92} \cite[Lect.~9]{Z_Zi}.
Theorem \ref{Z_thm:sphericity} was observed for the special
case $(p(t),q(t))=(t,t^2)$ and $a=1$ by Reiner \& Ziegler \cite{Z_ReZi}.
The more general Theorem~\ref{Z_thm:sphericity} follows from this 
by simple functoriality properties of the fiber polytopes: If
a polytope projection $\Delta_{n+2}\rightarrow Q$ is composed with
an affine transformation of the polygon~$Q$, the fiber polytope 
$\Sigma(\Delta_{n+2},Q)$ changes only by a multiplication by a constant factor. 
An affine transformation applied to the simplex $\Delta_{n+2}$ induces the same 
transformation on the fiber polytope.
\end{proof}

For the 1994 paper \cite{Z_ReZi}, the sphericity was discovered by chance,
and established by a simple algebraic verification (with computer algebra 
support), establishing that the length of the \emph{GKZ vector}  
(named after Gelfand, Kapranov and Zelevinsky)
\[
\text{GKZ}(T):= 
\sum_{i=1}^{n+3} \sum_{\sigma\in T\,:\, i\in\sigma}    \text{vol} (\sigma) f_i
\]
does not change under flips $T\rightarrow T'$. 
However, a ``geometric explanation'' was lacking then, and is still lacking now.
The true reason is still a mystery.

This is even more deplorable as the phenomenon occurs in other instances as well.
First, it does quite obviously extend to the realization of the
``permuto-associahedron'' that had been combinatorially described by
Kapranov \cite{Z_Kap}: Indeed, the length of the GKZ vector does not 
change under permutation of coordinates.

Moreover calculations (Ziegler 1994, unpublished) show there is a 
``secondary polytope like'' construction of the Bott--Taubes 
``cyclohedron'' \cite{Z_BottTaubes}
(also known as the ``type B'' generalized associahedron)
along quite similar lines, which again shows the same phenomenon:
It produces integer coordinates for the cyclohedron, with all
vertices on a sphere.
This can be verified in examples using e.g.\ \texttt{polymake} by Gawrilow \& Joswig \cite{Z_polymake},
it can be proved algebraically, but \emph{why} is it true?

\section{Realizing the multiassociahedron}
The boundary complex of the \emph{dual associahedron} is a simplicial
complex whose vertices correspond to diagonals of a convex polygon,
and whose faces correspond to subsets of non-crossing diagonals.
This complex can be naturally generalized to a beautiful family of simplicial
complexes with remarkable combinatorial properties. Members of this family are called 
\emph{simplicial multiassociahedra}.

Let $k\geq 1$ and $n\geq 2k+1$ be two positive integers. We say that a set 
of $k+1$ diagonals of a convex $n$-gon forms a \emph{$(k+1)$-crossing} 
if all the diagonals in this set are pairwise crossing. A diagonal is called 
\emph{$k$-relevant} if it is contained in some $(k+1)$-crossing, that is, if 
there are at least $k$ vertices of the $n$-gon on each side of the diagonal.
The \emph{simplicial multiassociahedron} $\Delta_{n,k}$ is the simplicial
complex of $(k+1)$-crossing-free sets of $k$-relevant diagonals of a convex 
$n$-gon. 
 
For example, the $2$-relevant diagonals of a convex $6$-gon 
(labeled as in the following figure, left picture) are  $14, 25$ and $36$,
	and the simplicial multiassociahedron $\Delta_{6,2}$ is the boundary complex
	of a triangle (right). The set of diagonals $\{14,25,36\}$ is not a face because
	they form a $3$-crossing.
\[
\includegraphics[scale=.6]{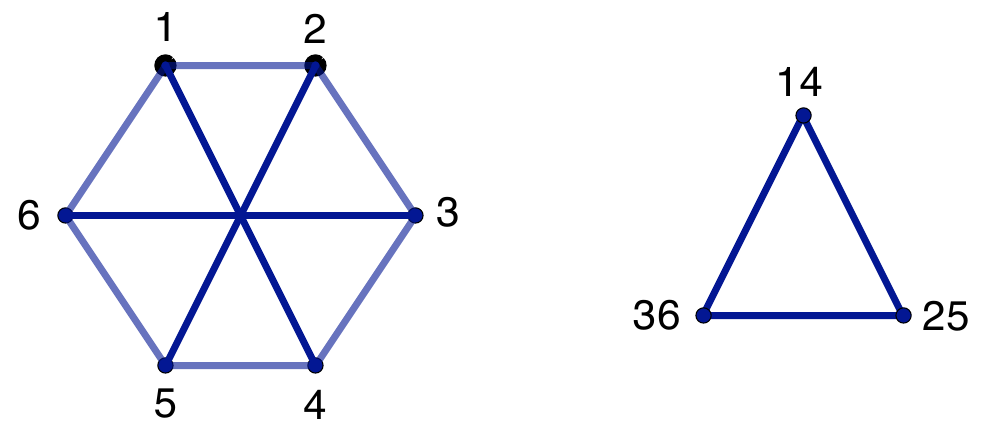}
\]

The vertices of the multiassociahedron $\Delta_{n,k}$ are given by $k$-relevant 
diagonals of the $n$-gon, and the facets correspond to \emph{$k$-triangulations},
that is, to maximal subsets of diagonals that do not contain any $(k+1)$-crossing.
For the case of $k=1$, the multiassociahedron is the simplicial complex of non-crossing
sets of diagonals, which coincides with the boundary complex of the dual associahedron.

The combinatorial structure of the multiassociahedron has been studied by several
authors. Apparently, it first appeared in work of Capoyleas \& Pach 
\cite{Z_capoyleas_a-turan_1992}, who showed that the maximal number of diagonals
in a $(k+1)$-crossing-free set is equal to $k(2n-2k-1)$. Nakamigawa 
\cite{Z_nakamigawa_a-generalization_2000} introduced the flip operation on 
$k$-triangulations and proved that the flip graph is connected.
Dress, Koolen \& Moulton \cite{Z_dress_on-line_2002} obtained a
reformulation of the Capoyleas--Pach result, and in particular proved that
all maximal $(k+1)$-crossing-free sets of diagonals have the same number of
diagonals. The results of Nakamigawa and Dress--Koolen--Moulton imply 
that the multiassociahedron $\Delta_{n,k}$ is a pure simplicial complex of
dimension $k(n-2k-1)-1$. A more recent approach for the study of $k$-triangulations,
using star polygons, was given by Pilaud \& Santos \cite{Z_pilaud_multitriangulations_2009}.
In 2003, Jonsson \cite{Z_jonsson_generalized_2003} showed that the multiassociahedron is a piecewise linear sphere. Then, he found an explicit $k\times k$ determinantal formula
of Catalan numbers counting the number of $k$-triangulations \cite{Z_jonsson_generalized_2005}. 
Additionally to the result of Jonsson about the multiassociahedron being a
topological sphere, Stump \cite{Z_stump_a-new_2011} proved that it is a
vertex-decomposable, and thus in particular shellable, simplicial sphere.
See also the results by Serrano \& Stump \cite{Z_serrano_maximal_2010}. 
 
All these results suggest that the multiassociahedron $\Delta_{n,k}$ could be 
realized as the boundary complex of a simplicial polytope of dimension $k(n-2k-1)$.
However, while for the classical associahedron we have many different construction
methods (see above), all the natural approaches seem
to fail for the multiassociahedron. 
The list of cases for which the multiassociahedron is known to be polytopal
is the following. The multiassociahedron $\Delta_{n,k}$ is the boundary 
complex of a:
\begin{compactitem}[$\circ$] 
	\item dual $(n-3)$-dimensional associahedron, if $k=1$;
	\item point, if $n=2k+1$;
	\item $k$-dimensional simplex, if $n=2k+2$;
	\item $2k$-dimensional cyclic polytope on $2k+3$ vertices, if $n=2k+3$ 		\cite{Z_pilaud_multitriangulations_2009};
	\item 6-dimensional simplicial polytope, if $n=8$ and $k=2$ 
		\cite{Z_bokowski_on-symmetric_2009}.   
\end{compactitem}  
  
Currently, the smallest open case is for $n=9$ and $k=2$.
Is there a simplicial polytope of dimension $8$ and $f$-vector
$(18, 153, 732, 2115, 3762, 4026, 2376, 594)$ which realizes 
the multiassociahedron $\Delta_{9,2}$?

Recently, the multiassociahedron has been generalized to a family of
vertex-decomposable simplicial spheres for finite Coxeter groups, 
by Ceballos, Labb\'e \& Stump \cite{Z_ceballos_subword_2011}.
They suggest a family of simple polytopes called \emph{generalized multiassociahedra}.
This family includes the generalized associahedra
\cite{Z_CFZ02} \cite{Z_hohlweg_permutahedra_2011}, and the  (simple) 
multiassociahedra of types $A$ and $B$ (see \cite{Z_soll_type-b_2009} for the type $B$ description). 
However, no polytopal realizations of generalized multiassociahedra
have been found except for the Coxeter groups of type $I_2(n)$, and for 
some particular cases in other types.
The (simple) generalized multiassociahedra of type $I_2(n)$ are given
by the duals of all even dimensional cyclic polytopes \cite{Z_ceballos_subword_2011}. 
Are there polytopal realizations for generalized multiassociahedra in general?

\providecommand{\bysame}{\leavevmode\hbox to3em{\hrulefill}\thinspace}
\providecommand{\MR}{\relax\ifhmode\unskip\space\fi MR }

\providecommand{\MRhref}[2]{%
  \href{http://www.ams.org/mathscinet-getitem?mr=#1}{#2}
}
\providecommand{\href}[2]{#2}

\end{document}